\documentclass{amsart}
\usepackage{amssymb}
\usepackage[curve,matrix,arrow]{xy}
\theoremstyle{plain}\newtheorem{Theorem}{Theorem}[section]
\theoremstyle{plain}
\theoremstyle{plain}\newtheorem{Corollary}[Theorem]{Corollary}
\theoremstyle{plain}\newtheorem{Lemma}[Theorem]{Lemma}
\theoremstyle{plain}\newtheorem{Proposition}[Theorem]{Proposition}
\theoremstyle{definition}\newtheorem{Definition}[Theorem]{Definition}
\theoremstyle{definition}
\theoremstyle{definition}
\theoremstyle{definition}\newtheorem{Remark}[Theorem]{Remark}

\def\CC{{\mathcal{C}}}
\def\CI{{\mathcal{I}}}

\def\P{{\mathbb{P}}}

\def\Z{{\mathbb{Z}}}

\def\End{\mathrm{End}}   \def\Endbar{\underline{\mathrm{End}}}

\def\Hom{\mathrm{Hom}}   \def\Hombar{\underline{\mathrm{Hom}}}
\def\ker{\mathrm{ker}}
\def\Id{\mathrm{Id}}
\def\Im{\mathrm{Im}}
\def\Ind{\mathrm{Ind}}

\def\mod{\mathrm{mod}}    \def\modbar{\underline{\mathrm{mod}}}

\def\rad{\mathrm{rad}}
\def\Res{\mathrm{Res}}
\def\soc{\mathrm{soc}}

\def\Ml{M}

\def\ten{\otimes}

\def\tenB{\otimes_B}

\title{On the graded center of the stable category of a finite $p$-group}
\author{Markus Linckelmann, Radu Stancu}

\date{\today}

\begin{document}

\begin{abstract} We show that for any finite $p$-group $P$ of rank at least $2$
and any algebraically closed field $k$ of characteristic $p$ the graded center
$Z^*(\modbar(kP))$ of the stable module category of finite-dimensional $kP$-modules has
infinite dimension in each odd degree, and if $p=2$ also
in each even degree. In particular, this provides examples of symmetric
algebras $A$ for which $Z^0(\modbar(A))$ is not finite-dimensional,
answering a question raised in \cite{Li2}.
\end{abstract}

\maketitle

\section{Introduction}

The graded center of a $k$-linear triangulated category $(\CC;\Sigma)$ over a commutative
ring $k$ is the graded $k$-module $Z^*(\CC) = Z^*(\CC;\Sigma)$ which in degree $n\in \Z$
consists of all $k$-linear natural transformations $\varphi : \Id_\CC \rightarrow
\Sigma^n$ satisfying
$\Sigma\varphi = (-1)^n\varphi\Sigma$; this becomes a graded
commutative $k$-algebra with multiplication essentially induced by composition in $\CC$.
We refer to \cite{Li2} for more details.
By \cite{Hap}, the stable category $\modbar(A)$ of finitely generated modules over a
finite-dimensional self-injective algebra $A$ is a triangulated category with shift
functor the inverse of the Heller operator. Its graded center has been calculated for
Brauer tree algebras \cite{KeLi} and in particular also uniserial algebras
\cite{KrYe}. These calculations suggest that Tate cohomlogy rings of blocks and the
graded centers of their stable module categories are closely related. Since the Tate
cohomology of a block is an invariant of the fusion system of the block, a good
understanding of graded centers might shed some light on the question to what extent the
fusion system of a block is determined by its stable module category. These calculations
also suggest that in order to determine the graded center of the stable module category
of a block one will need to do this first for a defect group algebra of the block, hence
for finite $p$-group algebras. This is what motivates the present paper. The following
result shows that $Z^0(\modbar(A))$ need not be finite-dimensional, answering a question
raised in \cite{Li2}.

\begin{Theorem} \label{2groupsstablecenter}
Let $P$ be a finite $2$-group of rank at least $2$ and $k$ an algebraically closed field
of characteristic~$2$. Evaluation at the trivial $kP$-module
induces a surjective homomorphism of graded $k$-algebras
$Z^*(\modbar(kP)) \longrightarrow \hat H^*(P;k)$
whose kernel $\CI$ is a nilpotent homogeneous ideal which is infinite-dimensional in each
degree; in particular,
$Z^0(\modbar(kP))$ has infinite dimension.
\end{Theorem}

For odd $p$ we have a slightly weaker statement:

\begin{Theorem} \label{pgroupsstablecenter}
Let $p$ be an odd prime, $P$  a finite $p$-group of rank at least
$2$ and $k$ an algebraically closed field of characteristic~$p$. Evaluation at the
trivial $kP$-module
induces a surjective homomorphism of graded $k$-algebras
$Z^*(\modbar(kP)) \longrightarrow \hat H^*(P;k)$
whose kernel $\CI$ is a nilpotent homogeneous ideal which is infinite-dimensional in each
 odd degree. \end{Theorem}

It is easy to see that the canonical
map $Z^*(\modbar(kP))\rightarrow \hat H^*(P;k)$ is surjective
with nilpotent kernel $\CI$ (see Lemma \ref{tatestablecenter} below). The point of the
above
theorem is that this kernel tends to have infinite dimension in
each degree (if $p=2$) and at least in each odd degree if $p>2$.
Note though that for $p$ odd there is no known example with
$Z^0(\modbar(kP))$ having infinite dimension.
The proof shows more precisely that these dimensions have as lower
bound the cardinality of the field $k$. Technically,
the proofs of the above theorems are based on the fact that for  $A$ a
symmetric algebras, an almost split sequence ending in an indecomposable
non projective $A$-module $U$ determines an {\it almost vanishing
morphism} $\zeta_U : U \rightarrow \Omega(U)$ which in turn provides elements of degree 
$-1$ in the graded center of the stable category; see e.g. \cite[Proposition 1.4]{Li2} or
section \ref{background}  below. Using modules with appropriate periods, these can then
be ``shifted" to all other degrees if the underlying characteristic is $2$ and all other
odd degrees if the characteristic is odd. The elements of $Z^*(\modbar(A))$ obtained in
this way will be
called {\it almost vanishing}; see \ref{avdef}  for details.
For Klein four groups, almost split sequences
turn out to be the only way to obtain elements in the graded center
of its stable module category beyond Tate cohomology. This can be seen using the
classification of indecomposable modules over Klein four groups and leads to a slightly
more precise statement.

\begin{Theorem} \label{kleinfour}
Let $P$ be a Klein four group and let $k$ be an algebraically closed field
of characteristic $2$. Then the evaluation at the trivial $kP$-module
induces a surjective homomorphism of graded $k$-algebras
$Z^*(\modbar(kP)) \longrightarrow \hat H^*(P;k)$
whose kernel $\CI$ is a homogeneous ideal which is infinite-dimensional in each degree.
Moreover, we have $\CI^2 = \{0\}$ and all elements in $\CI$
are almost vanishing. \end{Theorem}

This raises the question for which finite $p$-groups $P$ is
the graded center $Z^*(\modbar(kP))$ generated by Tate cohomology
and almost vanishing elements. Another interesting question underlying
some of the techical details below is the following. Given a periodic module
$U$ of period $n$ of a symmetric algebra $A$ over a field $k$, any
isomorphism $\alpha : U\cong \Omega^n(U)$ induces an algebra automorphism
of the stable endomorphism algebra $\Endbar_A(U)$ sending $\varphi$ to
$\alpha^{-1}\circ\Omega(\varphi)\circ\alpha$. When is this an inner
automorphism? Equivalently, when can $\alpha$ be chosen in such a way
that this automorphism is the identity? D. J. Benson \cite{Bencomm}
observed that the answer is positive if $\alpha$ is induced by an
element in Tate cohomology (or the Tate analogue of Hochschild cohomology)
because then $\alpha$ is the evaluation at $U$ of a natural transformation from the
identity functor on $\modbar(A)$ to the functor $\Omega^n$.

\section{Background material} \label{background}

Let $A$ be a finite-dimensional algebra over a field $k$. The stable module
category $\modbar(A)$ has as objects the finitely generated $A$-modules
and as morphisms, for any two finitely generated $A$-modules $U$, $V$,
the $k$-vector space $\Hombar_A(U,V) = \Hom_A(U,V)/\Hom_A^{pr}(U,V)$,
where $\Hom_A^{pr}(U,V)$ is the $k$-vector space of $A$-homomorphisms from
$U$ to $V$ which factor through a projective $A$-module, with composition
of morphisms induced by the usual composition of $A$-homomorphisms
in the category of finitely generated $A$-modules $\mod(A)$.
The Heller operator $\Omega_A$ sends an $A$-module $U$ to the kernel
of a projective cover $P_U\rightarrow U$ of $U$; this induces a
functor, still denoted $\Omega_A$, on $\modbar(A)$, which is unique up to unique
isomorphism of functors. If $A$ is self-injective then $\Omega_A$ is an equivalence,
and $\modbar(A)$ becomes a triangulated category with shift functor
$\Sigma_A = \Omega_A^{-1}$, sending an $A$-module to a cokernel of an
injective hull of $U$, with exact triangles induced by short exact
sequences in $\mod(A)$. See e.g. \cite{Hap} for details.
If $A$ is symmetric - that is, $A$ is
isomorphic, as $A$-$A$-bimodule to its $k$-dual $\Hom_k(A,k)$ -
then $A$ is in particular self-injective.
The literature on finite-dimensional algebras tends to privilege $\Omega_A$ from a
notational point of view, while the one on triangulated categories would rather use
$\Sigma_A$; we sometimes use both.
If $A$ is clear from the context and no confusion arises, we write $\Omega$ and $\Sigma$
instead of $\Omega_A$ and $\Sigma_A$, respectively.
If $B$ is a subalgebra of $A$ we have a canonical isomorphism
$\Omega_A^n(A\tenB V)\cong A\tenB \Omega_B^n(V)$ in $\modbar(A)$, where $V$ is a finitely
generated $B$-module and $n$ a positive integer.
If moreover $A$ is projective as $B$-module, we also have a canonical
isomorphism $\Omega_B^n(\Res^A_B(U))\cong\Res^A_B(\Omega_A^n(U))$ in $\modbar(B)$, where
$U$ is a finitely generated $A$-module, and the canonical adjunction isomorphism
$\Hom_A(A\tenB V,U)
\cong$ $\Hom_B(B,\Res^A_B(U))$ induces an isomorphism $\Hombar_A(A\tenB V,U)
\cong$ $\Hombar_B(V,\Res^A_B(U))$, where $U$ is an $A$-module and $V$ is a
$B$-module. The following well-known lemma states that the Heller translates commute
with the adjunction isomorphisms (we sketch a proof for the convenience of the reader).

\begin{Lemma}\label{OmegaAdjComm} Let $A$ be a finite-dimensional algebra over a field
$k$ and $B$ a subalgebra
such that $A$ is projective as $B$-module. Let $n$ be a positive integer. The
following diagram of canonical isomorphisms is commutative:
$$\xymatrix{
\Hombar_A(A\tenB\Omega_B^n(V),\Omega_A(U))\ar[r]_\simeq\ar[d]_\simeq
&\Hombar_B(\Omega_B^n(V),\Res^A_B(\Omega_A^n(U)))\ar[d]^\simeq\\
\Hombar_A(\Omega_A^n(A\otimes_B V),\Omega_A^n(U))
&\Hombar_B(\Omega_B^n(V),\Omega_B^n(\Res^A_B(U)))\\
\Hombar_A(A\otimes_B V,U)\ar[r]^\simeq\ar[u]^{\Omega_A^n}_{\simeq}
&\Hombar_B(V,\Res^A_B(U))\ar[u]_{\Omega_B^n}^{\simeq}
}$$
\end{Lemma}

\begin{proof}
We sketch the argument in the case $n=1$; the general case follows either
by induction, or directly with short exact sequences in the two diagrams below
replaced by exact sequences with $n+2$ terms of which all but possibly the first and last
are projective.
The Heller translates $\Omega_A$ and $\Omega_B$ are defined by the following two
commutative diagrams with exact rows, where $P_X$ denotes a projective cover of the
module $X$:
$$\xymatrix{
0\ar[r] &A\tenB\Omega_B(V)\ar[r]\ar[d]_\simeq &A\tenB P_V\ar[r]_\nu\ar[d]_\delta &A\tenB
V\ar[r]\ar@{=}[d] &0\\
0\ar[r] &\Omega_A(A\tenB V))\ar[r]\ar[d]_{\Omega_\alpha} &P_{A\otimes_B
V}\ar[r]\ar[d]_\gamma &A\otimes_B V\ar[r]\ar[d]_\alpha &0\\
0\ar[r] &\Omega_A(U)\ar[r] &P_U\ar[r]_\mu &U\ar[r] &0
}$$

$$\xymatrix{
0\ar[r] &\Omega_B(V)\ar[r]\ar[d]_{\Omega_\beta} &P_V\ar[r]_\rho\ar[d]_\phi &
V\ar[r]\ar[d]_\beta &0\\
0\ar[r] &\Omega_B(\Res_B^A(U))\ar[r]\ar[d]_\simeq &P_{\Res_B^A(U)}\ar[r]\ar[d]_\psi
&\Res_B^A(U)\ar[r]\ar@{=}[d] &0\\
0\ar[r] &\Res_B^A(\Omega_A U)\ar[r] &\Res_B^A(P_U)\ar[r]_\sigma &\Res_B^A(U)\ar[r]&0
}$$
Here $\mu$, $\rho$ are chosen surjective morphisms and
$\sigma = \Res^A_B(\mu)$, $\nu = \Id_A\ten \rho$.
We have that $\mu\circ\gamma\circ\delta=\alpha\circ\nu$. Since $\alpha$, $\beta$
correspond to each other through the adjunction
ismorphism, the two diagrams imply that the images of
$\Omega_\alpha$ and $\Omega_\beta$ correspond to each other through
the adjunction $\Hombar_B(\Omega_B(V),\Res^A_B(\Omega_A(U)))
\cong$ $\Hombar_A(A\tenB \Omega_B(V), \Omega_A(U))$. \end{proof}

A morphism $\alpha : U\rightarrow V$ in $\modbar(A)$ is called
{\it almost vanishing} if $\alpha\neq 0$ and if for any morphism
$\varphi : X \rightarrow U$ in $\modbar(A)$ which is not a split epimorphism we have
$\alpha\circ\varphi = 0$. Since endomorphism
algebras of indecomposable $A$-modules are (not necessarily split) local, it follows from
\cite[I.4.1]{Hap} that this condition is equivalent to its dual; that is, for any
morphism $\psi : V \rightarrow Y$ in $\modbar(A)$ which is not a
split monomorphism we have $\psi\circ\alpha = 0$. In order to follow almost vanishing
morphisms through the standard
adjunction isomorphisms we collect a few elementary statements
on split epimorphisms. In statements and proofs
where we consider both $A$-homomorphisms and their classes in the stable
category we adopt the notational convention that if $\varphi : U\to V$
is a homomorphism of $A$-modules then $\underline\varphi$ denotes the
class of $\varphi$ in $\Hombar_A(U,V)$. As in any triangulated category,
epimorphisms in $\modbar(A)$ are split epimorphisms, and those are
essentially induced by split epimorphisms in $\mod(A)$:

\begin{Lemma} \label{stablesplitEpi}
Let $A$ be a finite-dimensional $k$-algebra and $\varphi :  U\rightarrow
V$ an $A$-homomorphism. Suppose that $V$ is indecomposable non projective.
Then $\varphi$ is a split epimorphism in $\mod(A)$ if and only if
its class $\underline\varphi : U\rightarrow V$ is a split epimorphism in $\modbar(A)$.
\end{Lemma}

\begin{proof}
Suppose $\underline\varphi$ is a split epimorphism in $\modbar(A)$.
Let $\sigma:V\rightarrow U$ be an $A$-homomorphism such that
$\underline\varphi\circ\underline\sigma = \underline\Id_V$.
Thus $\varphi\circ\sigma-\Id_V$ factors through a projective
module. Since $V$ is indecomposable non projective this implies that
$\varphi\circ\sigma-\Id_V\in J(\End_A(V))$, hence $\rho = \varphi\circ\sigma$ is an
automorphism of $V$ and $\sigma\circ\rho^{-1}$ a section for $\varphi$. The converse is
trivial.
\end{proof}

\begin{Lemma}\label{splitEpi}
Let $A$ be a $k$-algebra and $B$ a subalgebra of $A$. Suppose that $B$ has a complement
in $A$ as a $B$-$B$-bimodule. Let $\varphi:V\to W$ be a homomorphism of left $B$-modules
and set
$\psi = \Id_A\ten\varphi : A\tenB V\rightarrow A\tenB W$.
Then $\psi$ is a split epimorphism if and only of $\varphi$ is a split epimorphism.
\end{Lemma}

\begin{proof}
Suppose that $\psi$ is a split epimorphism; that is, there is
an $A$-homomorphism $\sigma : A\tenB W\rightarrow A\tenB V$ satisfying $\psi\circ\sigma =
\Id_{A\tenB W}$. Let $C$ be a complement of $B$ in $A$ as a $B$-$B$-bimodule; that is,
$A = B\oplus C$ as a $B$-$B$-bimodule.
Denote by $\pi : A\tenB V\rightarrow V$ and $\tau : A\tenB W\rightarrow W$
the canonical projections with kernel $C\tenB V$ and $C\tenB W$,
respectively, induced by the projection $A\rightarrow B$ with kernel $C$. Explicitly, for
$v\in V$ and $a\in A$
we have $\pi(a\ten v) = a\ten v$ if $a\in B$ and $\pi(a\ten v) = 0$ if
$a\in C$; similarly for $\tau$. Since $\psi = \Id_A\ten \varphi$
we have $\varphi\circ\pi = \tau\circ\psi$. Define $\rho : W\rightarrow V$
by $\rho(w) = \pi(\sigma(1\ten w))$, for any $w\in W$. Then
$\varphi(\rho(w)) =$ $\varphi(\pi(\sigma(1\ten w))) = $ $\tau(\psi(\sigma(1\ten w))) =$ $
\tau(1\ten w) = w$, hence $\varphi$ is
a split epimorphism with section $\rho$. The converse is trivial.
\end{proof}

\begin{Lemma} \label{Omegainvariance}
Let $A$ be a finite-dimensional algebra over an algebraically closed
field $k$ and $B$ a subalgebra such that $A$ is projective as left and right $B$-module
and such that $B$ has
a complement in $A$ as a $B$-$B$-bimodule. Let $V$ be a finitely
generated indecomposable non projective $B$-module such that $U=A\tenB V$ is
indecomposable. Then $U$ is non projective.
Let $\zeta_V : V\rightarrow \Omega_B(V)$ and $\zeta_U : U\rightarrow \Omega_A(U)$ be
almost vanishing homomorphisms
in $\modbar(B)$ and $\modbar(A)$, respectively.
Suppose that there is an isomorphism $\beta : \Omega_B^n(V)\cong V$ in $\modbar(B)$ such
that $\Omega_B(\beta)\circ\Omega_B^n(\zeta_V)=(-1)^n\zeta_V\circ\beta$.
Then $\Omega_A^n(U)\cong U$ and for any isomorphism
$\alpha : \Omega_A^n(U)\cong U$ in $\modbar(A)$ we have
the $\Omega_A(\alpha)\circ\Omega_A^n(\zeta_U)=(-1)^n\zeta_U\circ\alpha$.
\end{Lemma}

\begin{proof}
We first point out where we use the hypothesis on $k$ being algebraically closed.
If we can find {\it some} isomorphism
$\alpha : \Omega_A^n(U)\cong U$ in $\modbar(A)$ such that
$\Omega_A(\alpha)\circ\Omega_A^n(\zeta_U)=(-1)^n\zeta_U\circ\alpha$ then
this holds for {\it any} isomorphism because $\End_A(U)$ is split local.
Moreover, $\zeta_U$ is unique up to a nonzero scalar.
Similarly for $V$ and $\zeta_V$.
By the assumptions, $A = B\oplus C$ for some $B$-$B$-submodule
$C$ of $A$ which is finitely generated projective as left and right
$B$-module. Thus $\Res^A_B$ sends any projective $A$-module to a
projective $B$-module, and $V$ is a direct summand of $\Res^A_B(U)$.
Since $V$ is non projective this implies that $U$ is non projective.
By adjunction we have
$$\Hombar_{A}(U,U) \cong \Hombar_{B}(V,\Res^A_B(A\tenB V))$$ We prove that the image
$\eta$ of $\zeta_U$ under this adjunction is a morphism
which, when composed with any projection of $\Res^A_B(A\tenB V)$ onto an
indecomposable direct summand is either zero or almost vanishing in $\modbar(B)$. This is
equivalent to showing that $\eta$ precomposed with any morphism
ending at $V$ which is not a split epimorphism in $\modbar(B)$ yields zero.
Let $X$ be a $B$-module and $\varphi\in\Hombar_B(X,V)$ such that the image
$\underline\varphi$ in $\Hombar_B(X,V)$ is not a split epimorphism.
Then by Lemma \ref{stablesplitEpi} and Lemma \ref{splitEpi} the morphism $\psi=\Id_A\ten
\varphi$ in $\Hombar_A(A\tenB X,U)$ is not a split epimorphism and hence
$\zeta_M\circ\psi$ is zero in $\Hombar_{A}(A\tenB X,M)$. Applying the adjunction  to
$\zeta_M\circ\psi$ we get that $\eta\circ\varphi$ is zero in
$\Hombar_{B}(X,\Res^A_B(A\tenB V)))$. Thus all components of $\eta$
to the indecomposable direct summands of $\Res^A_B(A\tenB V))$ are
either zero or a scalar multiple of $\zeta_V$.
The hypothesis on $\zeta_V$ implies that $\Omega^n_{B}(\eta)=(-1)^n\eta$, modulo
appropriate identifications.
It follows from Lemma \ref{OmegaAdjComm} that $\Omega^n(\zeta_U)=(-1)^n\zeta_U$,
again modulo appropriate identifications, which proves the result.
\end{proof}

The existence of
an almost vanishing morphism $\alpha : U\rightarrow V$ in $\modbar(A)$
forces that $U$, $V$ are indecomposable non projective $A$-modules,
and that the exact triangle in $\modbar(A)$ of the form
$$\xymatrix{\Omega(V) \ar[r] & E \ar[r] & U \ar[r]^{\alpha} & V }$$ is {\it almost split}
in the sense of \cite[I.4.1.]{Hap}, hence induced
by an almost split exact sequence of $A$-modules of the form
$$\xymatrix{0 \ar[r] & \Omega(V) \ar[r] & E \ar[r] & U \ar[r] & 0 }$$
which in turn forces that $V\cong \Omega(U)$
as $A$ is symmetric (cf. \cite[4.12.8]{Ben}). This shows that if there is an almost
vanishing
morphism $U \rightarrow \Sigma^n(U)$ in $\modbar(A)$ for some integer $n$
then $\Sigma^n(U)\cong \Sigma^{-1}(U)$, and hence either $n = -1$ or $U$ is periodic of
period dividing $n+1$. As mentioned earlier, almost vanishing
morphisms define elements in the graded center. It is convenient to extend the
terminology of almost vanishing morphisms to elements in the graded center:

\begin{Definition} \label{avdef}
Let $A$ be a symmetric algebra over a field $k$.
For any integer $n$ we say that an element
$\varphi\in Z^n(\modbar(A))$ is {\it almost vanishing} if for any
indecomposable non projective $A$-module $U$ either $\varphi(U) = 0$
or $\varphi(U) : U \rightarrow \Sigma^n(U)$ is almost vanishing.
An element in $Z^*(\modbar(A))$ is called {\it almost vanishing}
if all of its homogeneous components are almost vanishing.
We denote by $\CI_A$ the set of all almost vanishing elements
in $Z^*(\modbar(A))$.
\end{Definition}

This definition makes sense for arbitrary triangulated categories.
Note that the zero element of $Z^*(\modbar(A))$ is
contained in $\CI_A$; this slight departure from the definition of
almost vanishing morphisms has the following obvious advantage:

\begin{Lemma} \label{avideal}
Let $A$ be a symmetric algebra over a field $k$. Then $\CI_A$ is an ideal in
$Z^*(\modbar(A))$ whose square is zero. Moreover, $\CI_A$ annihilates
every ideal consisting of nilpotent elements in $Z^*(\modbar(A))$.
\end{Lemma}

\begin{proof} If $\varphi\in Z^n(\modbar(A))$ is nilpotent, where $n$
is an integer, then for any indecomposable non projective $A$-module
$U$ the morphism $\varphi(U) : U\rightarrow \Sigma^n(U)$ is not an
isomorphism. Since $U$, $\Sigma^n(U)$ are both indecomposable,
$\varphi(U)$ is neither a split epimorphism nor a split monomorphism.
Thus $\varphi(U)$ composed with any almost vanishing
morphism yields zero. The result follows.
\end{proof}

In the statements of Theorem \ref{pgroupsstablecenter} and
Theorem  \ref{kleinfour}, in order to show that $Z^*(\modbar(A))$ is not
finite-dimensional in a particular degree we actually show that $\CI_A$ is not
finite-dimensional in that degree. We will need the following lemma, which is well-known
and strictly analogous to  \cite[Proposition 1.3]{Li2}, for instance.

\begin{Lemma} \label{tatestablecenter}
Let $p$ be a prime, $P$ a finite $p$-group and $k$ a field of
characteristic $p$. Evaluation at the trivial $kP$-module $k$
induces a surjective graded $k$-algebra homomorphism $Z^*(kP)\rightarrow \hat H^*(P;k)$
whose kernel $\CI$ is a
nilpotent ideal.
\end{Lemma}

\begin{proof}
This is the same proof as that of \cite[Proposition 1.3]{Li2},
with $H^*(P;k)$, $HH^*(kP)$, $D^b(kP)$  replaced by Tate cohomology $\hat H^*(P;k)$, the
Tate analogue of Hochschild cohomology
$\hat{HH}^*(kP)$ and $\modbar(kP)$, respectively.
\end{proof}

\section{On some modules of period $1$ and $2$}

In order to construct elements in a given degree $n$ of the graded center of the stable
category of a symmetric algebra $A$ over a field $k$ it is not sufficient to construct
$A$-modules with period dividing $n+1$; we also need to have some control over the effect
of the shift on endomorphisms in order to verify the additional compatibility condition
with the shift functor of natural transformations belonging to $Z^n(\modbar(A))$. This is
the purpose of the elementary observations in this section. We use without further
comment that if $k$ has positive characteristic
$p$ and if $x\in Z(A)$ such that $x^{p-1}\neq 0$, $x^p = 0$, then
$1+x$ is an invertible element in $A$ of order $p$ and generates
a subalgebra which can be identified with the truncated polynomial
algebra $k[x]/(x^p)$ or the group algebra $k\langle 1+x\rangle$,
whichever is more convenient. The algebra $A$ may or may not be projective as $k\langle
1+x\rangle$-module; this will depend on $x$. If $A = kP$ is the group algebra of an
elementary abelian $p$-group $P$ with
minimal generating set $\{u_1,u_2,..,u_r\}$, then any nonzero $k$-linear
combination $x = \sum_{i=1}^r\ \lambda_i(1-u_i)$ has the property that
$A$ is projective as $k\langle 1+x\rangle$-module. The subgroups
generated by $1+x$ for those $x$ are called cyclic shifted subgroups of~$A$.

\begin{Lemma} \label{cyclicmodify}
Let $A$ be a finite-dimensional algebra over a field $k$
of prime characteristic $p$ and $x$ an element in $A$ such
that $x^p=0$, $x^{p-1}\neq 0$ and such that $A$ is projective as a
$k\langle 1+x\rangle$-module. Let $u\in A^\times$ such that $ux = xu$.
Then $(ux)^p=0$, $(ux)^{p-1}\neq 0$ and $A$ is projective as a
$k\langle 1+ux\rangle$-module.
\end{Lemma}

\begin{proof} Since $x$ and $u$ commute we clearly have $(ux)^p=0$ and $(ux)^{p-1}\neq
0$. The hypothesis that $A$ is projective
as  $k\langle 1+x\rangle$-module is equivalent to $\dim_k(x^{p-1}A) =
\frac{1}{p}\dim_k(A)$. As $u$ is invertible and commutes to $x$, we have $(ux)^{p-1}A =
x^{p-1}A$, and hence $A$ is projective as
$k\langle 1+ux\rangle$-module.
\end{proof}

\begin{Lemma} \label{periodonemodules}
Let $A$ be a finite-dimensional algebra over a field $k$ of prime characteristic $p>0$.
Let $x\in A$ such that $x^{p-1}\neq 0$,  $x^p = 0$ and such that $A$ is free as right
$k\langle 1+x\rangle$-module.  Set $M = Ax^{p-1}$.

\smallskip\noindent (i)
If $p=2$ there is an isomorphism $\alpha : M\cong \Omega(M)$ in $\modbar(A)$
such that $\underline\varphi\circ\alpha = \alpha\circ\Omega(\underline\varphi)$
for any $\varphi\in\End_A(M)$.

\smallskip\noindent (ii)
If $p>2$ there is an isomorphism $\beta : M\cong \Omega^2(M)$ in $\modbar(A)$
such that $\underline\varphi\circ\beta = \beta\circ\Omega^2(\underline\varphi)$
for any $\varphi\in\End_A(M)$.

\smallskip\noindent (iii)
We have $\Endbar_A(M) \cong \End_A(M)$.
\end{Lemma}

\begin{proof}
Since $A$ is free as right $k\langle 1+x\rangle$-module
we have an exact sequence of $A$-modules of the form
$$\xymatrix{ 0 \ar[r] & M \ar[r] & A \ar[r]^{\epsilon} & A \ar[r]^{\pi} & M\ar[r] &0
}$$
where $\epsilon$ is given by right multiplication with $x$ and $\pi$
is given by right multiplication with $x^{p-1}$. Thus for this choice
of projective covers we have $\Omega(M) = Ax$ and $\Omega^2(M) = M$.
Let $\varphi\in\End_A(M)$. Then $\varphi$ lifts to an endomorphism
$\psi\in\End_A(A)$ satisfying $\pi\circ\psi = \varphi\circ\pi$.
Let $y\in A$ such that $\psi(a) = ay$ for all $a\in A$. Thus
$\varphi(ax^{p-1}) = ayx^{p-1} = ax^{p-1}y$ for all $a\in A$.
Also, since $x$ is central we have $\epsilon\circ\psi = \psi\circ\epsilon$.
Thus $\psi$ restricts to the endomorphism $\varphi$ of $M$, which
shows (i) and (ii). Suppose now that $\varphi$ factors through a projective $A$-module.
Then $\varphi$ factors through the inclusion $M\subseteq A$, or equivalently, there is
$y\in A$ such that $\varphi(m) = my$ for all $m\in M$ and such that $Ay\subseteq M$.
But then in particular $y\in M = Ax^{p-1}$, hence
$y = ux^{p-1}$ for some $u\in A$. Since any $m\in M$ is of the form $m = ax^{p-1}$ for
some $a\in A$ we get that $\varphi(m) = \varphi(ax^{p-1}) = ax^{p-1}ux^{p-1} =
aux^{2(p-1)} = 0$,
whence (iii).
\end{proof}

The next observation shows that if the center of a finite $p$-group
$P$ has rank at least $2$ then the construction principle for modules of period $1$ and
$2$ over the group algebra $A = kP$ in the
two previous lemmas yields infinitely many isomorphism classes
whenever the field $k$ is infinite.

\begin{Proposition} \label{isomod}
Let $p$ be a prime, $P$ a finite $p$-group and $k$ a field of
characteristic $p$. Set $A = kP$. Suppose that $Z(P)$ has an
elementary abelian subgroup $E = \langle s\rangle \times \langle t\rangle$ of rank $2$,
for
some elements $s, t\in Z(P)$ of order $p$. Let $(\lambda,\mu), (\lambda',\mu')\in
k^2-\{(0,0)\}$. Set
$x = \lambda(s-1) + \mu(t-1)$ and $x' = \lambda'(s-1) + \mu'(t-1)$.
Then $Ax$ is absolutely indecomposable as left $A$-module, and
the following statements are equivalent.

\smallskip\noindent (i)
$Ax \cong Ax'$ as left $A$-modules.

\smallskip\noindent (ii)
$Ax = Ax'$.

\smallskip\noindent (iii)
The images of $(\lambda,\mu)$ and $(\lambda',\mu')$ in $\P^1(k)$
are equal.
\end{Proposition}

\begin{proof}
Since $A = kP$ is split local, any quotient of $A$ as $A$-module
is absolutely indecomposable.
Suppose (i) holds. Since $A$ is symmetric, any isomorphism
$Ax\cong Ax'$ extends to an endomorphism of $A$,
which is given by right multiplication with an element $y\in A$.
In particular, $Axy = Ax'$. But $x\in Z(A)$ implies $Axy = Ayx
\subseteq Ax$. Thus $Ax'\subseteq Ax$. Exchanging the roles of $x$
and $x'$ shows $Ax = Ax'$, whence (ii) holds. The converse is
trivial. Suppose now that (ii) holds. Arguing by contradiction,
suppose that the images of $(\lambda,\mu)$, $(\lambda',\mu')$ in
$\P^1(k)$ are different, or equivalently, that $(\lambda,\mu)$, $(\lambda',\mu')$ are
linearly independent in $k^2$. Since $Ax = Ax'$ contains any $k$-linear combination of
$x$ and $x'$ it
follows that $Ax$ contains any $k$-linear combination of $s-1$ and
$t-1$. But then $A(s-1) = Ax = A(t-1)$ because all involved modules
have the same dimension $\frac{p-1}{p}\dim_k(A)$ as $A$ is projective
as right $kE$-module. However, clearly $kE(s-1)\neq kE(t-1)$, so
this is impossible. This shows that (ii) implies (iii). The converse
is again trivial.
\end{proof}

\begin{Corollary} \label{isomodcor}
Let $p$ be a prime, $P$ a finite $p$-group of rank at least $2$ and
$k$ an infinite field of characteristic $p$.
If $p=2$ then $kP$ has infinitely many isomorphism classes of absolutely
indecomposable modules of dimension $\frac{|P|}{2}$ whose period is $1$,
and if $p>2$ then $kP$ has infinitely many isomorphism classes of absolutely
indecomposable modules of dimension $\frac{|P|}{p}$ whose period is $2$.
\end{Corollary}

\begin{proof}
Let $E$ be an elementary abelian subgroup of rank $2$ of $P$.
By Proposition \ref{isomod} and Lemma \ref{periodonemodules}, the result holds for $kE$
instead of $kP$.  Induction $\Ind^P_E$ is an exact
functor, preserving periodicity, and by Mackey's formula, for any indecomposable
$kP$-module $M$ there are, up to isomorphism, only finitely many indecomposable
$kE$-modules $N$ satisfying $\Ind^P_E(N) \cong M$, which implies the  statement.
\end{proof}

\section{Proof of Theorem \ref{2groupsstablecenter} and 
Theorem \ref{pgroupsstablecenter}}                     

Let $p$ be a prime and $P$ a finite $p$-group of
rank at least $2$. Let $k$ be an algebraically closed
field of characteristic $p$. By Lemma \ref{tatestablecenter}, evaluation at the trivial
$kP$-module induces a surjective graded algebra homomorphism $Z^*(\modbar(kP))\rightarrow
\hat H^*(P;k)$
with nilpotent kernel $\CI$. Let $E = \langle s\rangle \times \langle t\rangle$ be an
elementary abelian subgroup of $Z(P)$ of rank $2$ for some elements
$s$, $t$ in $Z(P)$ of order $P$ and let $(\lambda,\mu)\in k^2-\{(0,0)\}$
Set $x = \lambda(s-1) + \mu(t-1)$ and $M = Ax^{p-1}$.
By Lemma \ref{periodonemodules} the module $M$ has
period $1$ if $p=2$ and period $2$ if $p$ is odd.
The almost split exact sequence ending in $M$ is of the form
$$\xymatrix{ 0 \ar[r] & \Omega^2(M) \ar[r] & E \ar[r] & M \ar[r] & 0}$$
hence represented by
an almost vanishing morphism $\zeta_M : M \rightarrow \Omega(M)$.
By \cite[Proposition 1.4]{Li2}, this defines a natural transformation $\varphi_M : \Id
\rightarrow \Omega$ as follows.
Set $\varphi_M(M) = \zeta_M$, set $\varphi_M(N) = 0$ if $N$ is indecomposable
non projective and not isomorphic to $M$ or $\Omega(M)$, and in the case
$p$ odd, set $\varphi_M(\Omega(M)) = -\Omega(\zeta_M)$.
Then, thanks to Lemma \ref{periodonemodules} again,
$\varphi$ is indeed a natural transformation belonging to
$Z^{-1}(\modbar(kP))$. If $p=2$ (resp. $p > 2$)then for any integer $n$ (resp. even
integer $n$) we have
$M \cong \Sigma^n(M)$, and hence $\varphi$ determines also an
element in $Z^{n-1}(\modbar(kP))$ for any integer $n$ (resp.
even integer $n$). By Lemma
\ref{isomod}, this implies that $Z^{n-1}(\modbar(kP))$ is infinite
dimensional for any integer $n$ if $p=2$ and for any even integer
if $p>2$. This proves Theorem \ref{2groupsstablecenter}
and Theorem \ref{pgroupsstablecenter}.

\begin{Remark}
D. J. Benson \cite{Bencomm} pointed out that it is easy to construct indecomposable
modules for finite $p$-group algebras with period
one for odd  $p$. One might therefore wonder whether
the stronger statement of Theorem \ref{2groupsstablecenter}
holds for $p$ odd as well. \end{Remark}

\section{Proof of Theorem \ref{kleinfour}}

Let $k$ be an algebraically closed field of characteristic $2$. Denote by $V_4 = \langle
s, t\rangle$ a Klein four group, with commuting involutions $s, t$. By results of
Ba\v{s}ev \cite{Bas}, Heller and Reiner \cite{HeRe} (see also \cite[4.3.3]{Ben} for an
expository account), for any positive integer $n$
and any $\lambda \in k \cup \{\infty\}$ there is, up to isomorphism, a unique
indecomposable $kV_4$-module $M^\lambda_n$ of dimension $2n$,
having a $k$-basis $\{a_1, a_2,\ldots,a_n, b_1,b_2,\ldots, b_n\}$ on which the
elements $x = s-1$ and $y = t-1$ act by $$xa_i=b_i\ (1\leq i\leq n),\ ya_i = \lambda b_i
+b_{i+1}\ (1\leq i <n),\ ya_n = \lambda b_n\ ,xb_i = 0 = yb_i\ (1\leq i\leq n)$$
provided that $\lambda\in k$, and by
$$xa_i = b_{i+1}\ (1\leq i<n),\ xa_n=0,\ ya_i=b_i\ (1\leq i\leq n),\ ya_n = b_n\ ,xb_i =
0 = yb_i\ (1\leq i\leq n)$$
if $\lambda = \infty$. Then $M^\lambda_n$, with $\lambda\in k\cup\{\infty\}$, is
a set of representatives of the isomorphism classes of indecomposable
$kV_4$-modules of dimension $2n$. This set can be indexed by the elements
in $\P^1(k)$ with $M^\lambda_n$ corresponding to the image of $(\lambda,1)$
in $\P^1(k)$ if $\lambda\in k$, and with $M^\infty_n$ corresponding to
the image of $(0,1)$ in $\P^1(k)$. The natural action of $GL_2(k)$ on
$\P^1(k)$ translates to a transitive action of $GL_2(k)$ on the set of
isomorphism classes of $2n$-dimensional indecomposable $kV_4$-modules.
The group $GL_2(k)$ can be identified with a group of $k$-algebra automorphisms
of $kV_4$, with $g = \left(\begin{array}{cc} \alpha & \beta\\ \gamma &
\delta\end{array}\right)$
acting on $kV_4$ by $g(x) = \alpha x+\beta y$ and $g(y) = \gamma x+\delta y$.
In this way, any $kV_4$-module $M$ gives rise via restriction along the
automorphism $g$ to a $kV_4$-module ${^gM}$, which
is equal to $M$ as $k$-vector space, with $z\in kV_4$ acting as $g(z)$ on $M$.
This defines another action of $GL_2(k)$ on the set of isomorphism classes of
$2n$-dimensional indecomposable $kV_4$-modules, and, of course, these actions
are compatible: one checks that for $\lambda\in k$ we have
$M^\lambda_n\cong {^g(M^0_n)}$, where
$g = \left(\begin{array}{cc} 1 & 0 \\ \lambda & 1 \end{array}\right)$,
and that $M^\infty_n \cong {^h(M^0_n)}$, where $h = \left(\begin{array}{cc} 0 & 1 \\ 1 &
0 \end{array}\right)$.
The point of these remarks is that in order to calculate (stable)
endomorphism algebras of the modules $M^\lambda_n$, identify almost vanishing
morphisms and check their invariance under the Heller operator, it
suffices to consider the case $\lambda = 0$.
We collect some basic facts on endomorphisms of the modules $M_n^\lambda$,
many of which are, of course, well-known (such as the fact that $M_n^\lambda$
has period $1$). In what follows we set $M_n = M_n^0$, for any positive
integer $n$.

\begin{Lemma}\label{zeroendomorphisms}
Let $n$ be a positive integer and denote by $\{a_1,a_2,..,a_n,b_1,b_2,..,b_n\}$ a
$k$-basis
of $M_n$ such that $xa_i = b_i$ for $1\leq i\leq n$,
$ya_i = b_{i+1}$ for $1\leq i < n$, $ya_n = 0$ and $xb_i = 0= yb_i$
for $1\leq i\leq n$. Let $\varphi\in\End_{kV_4}(M_n)$.

\smallskip\noindent (i)
We have $\Omega(M_n) \cong M_n$.

\smallskip\noindent (ii)
If $\varphi\in\End_{kV_4}^{pr}(\Ml_n)$ then $\Im(\varphi)\subseteq \soc(\Ml_n)$.

\smallskip\noindent (iii)
If $\Im(\varphi) = k\cdot b_1$ then $\varphi\not\in \End_{kV_4}^{pr}(M_n)$.

\smallskip\noindent (iv) For $1\leq i\leq n$ denote by $\varphi^n_i$ the endomorphism of
$M_n$
sending $a_i$ to $b_1$ and all other basis elements of $M_n$ to zero.
Then the set $\{\underline\varphi^n_i\}_{1\leq i\leq n}$ is linearly
independent in $\Endbar_{kV_4}(M_n)$.

\smallskip\noindent (v)
We have $\dim_k(\End_{kV_4}^{pr}(M_n)) = n^2-n$.
\end{Lemma}

\begin{proof}
For $1\leq i\leq n$ denote by $s_i$ the element in
$(kV_4)^n$ whose $i$-th component is $1$ and whose other componenets are $0$.
Then $\{s_1,s_2,..,s_n\}$ is a basis of the free $kV_4$-module $(kV_4)^n$.
There is an injective $kV_4$-homomorphism
$$\iota : M_n \longrightarrow (kV_4)^n$$
such that $\iota(a_i) = ys_i+xs_{i+1}$ for $1\leq i < n$ and $\iota(a_n) = ys_n$, and
there is a surjective $kV_4$-homomorphism
$$\pi : (kV_4)^n \longrightarrow M_n$$
such that $\pi(s_i) = a_i$ for $1\leq i\leq n$. One checks that $\Im(\iota) = \ker(\pi)$,
which shows that $\Omega(M_n)
\cong M_n$, whence (i). If $\varphi\in \End_{kV_4}^{pr}(M_n)$,
the inclusion $\Im(\varphi)\subset \rad(M_n)$ is a general fact, and
we have $\rad(M_n) = \soc(M_n)$, which shows (ii). Since $\varphi$ factors through a
projective module, it factors through both
$\iota$ and $\pi$; thus there is $\psi\in\End_{kV_4}((kV_4)^n)$ such that
$\varphi = \pi\circ\psi\circ\iota$. For $1\leq i, j \leq n$ let
$\alpha_{i,j}\in k$ and $r_{i,j}\in \rad(kV_4) = xkV_4+ykV_4$,
viewed as element in the $j$-th component of $(kV_4)^n$, such that
$$\psi(s_i) = \sum_{j=1}^n\ \alpha_{i,j}\ s_j\ +\ r_{i,j}$$
Since $\Im(\iota) \subset \rad((kV_4)^n)$ one easily sees that the
we may assume $r_{i,j} = 0$. Two short verifications show that
$$\varphi(a_n) = \sum_{j=2}^n\ \alpha_{n,j-1}\ b_j$$
and, for $1\leq i < n$,
$$\varphi(a_i) = \alpha_{i+1, 1} b_1\ +\ \sum_{j=2}^n\
(\alpha_{i,j-1}+\alpha_{i+1,j})b_j$$
which also shows again that $\Im(\varphi)\subseteq\soc(M_n)$.
Moreover, this calculation shows that if $\Im(\varphi)\subseteq kb_1$,
then the equation for $\varphi(a_n)$ implies $\alpha_{n,j} = 0$
for $1\leq j < n$, and the equations for the $\varphi(a_i)$ then
imply inductively that $\alpha_{i,j} = 0$ for $1\leq j<i\leq n$, which
in turn yields $\varphi = 0$. This proves (iii). Statement (iv) is an
immediate consequence of (iii). For any pair $(i,j)$, with $1\leq i,j\leq n$
there is a unique endomorphism of $M_n$ sending $a_i$ to $b_j$ and all
other basis elements to zero. Thus the space of endomorphisms of $M_n$
whose image is contained in $\soc(M_n)$ has dimension $n^2$. It follows
form (iv) that $\dim_k(\End_{kV_4}^{pr}(M_n)) \leq n^2-n$.
Setting $\varphi_{i,j} = \pi\circ\psi_{i,j}\circ\iota$, where $\psi_{i,j}$
is the unique endomorphism determined by the coefficients $\alpha_{i,j} = 1$
and $\alpha_{r,s} = 0$ for $(r,s) \neq (i,j)$ an easy verification
shows that the set
$\{\varphi_{i,j}\ |\ 1\leq i\leq n,\ 1\leq j < n\}$ is linearly
independent in $\End_{kV_4}^{pr}(M_n)$, whence the equality in (v).
\end{proof}

\begin{Lemma} \label{morphisms}
For any positive integer $n$ denote by
$\{\alpha^n_1,\alpha^n_2,..,\alpha^n_n,\beta^n_1,\beta^n_2,..,\beta^n_n\}$
a $k$-basis of $M_n$ such that $xa^n_i = b^n_i$ for $1\leq i\leq n$,
$ya^n_i=b^n_{i+1}$ for $1\leq i\leq n-1$, $ya^n_n = 0$, and such
that $xb^n_i = 0 = yb^n_i$ for $1\leq i\leq n$.
For any positive integers $m<n$ there is a homomorphism
$\zeta_m^n:M_m\to M_n$ such that $\zeta(a^m_i) = a^n_{n-m+i}$ for
$1\leq i\leq m$, and a homomorphism $\xi^n_m : M_n\to M_m$ such that
$\xi_m^n(a^n_i)=a^m_i$  for $1\le i\le m$  and $\xi_m^n(a^n_i)=0$ for $m+1\le i\le n$.
Then the sequence
$$\xymatrix{ 0 \ar[r] & \Ml_m \ar[r]^{\zeta_m^n} & \Ml_n \ar[r]^{\xi_{n-m}^n} &
\Ml_{n-m}\ar[r] &0 }$$
is exact. In particular $\zeta_m^n$ is a non-split monomorphism and $\xi_m^n$ is a 
non-split epimorphism.
\end{Lemma}

\begin{proof} Straightforward verification.
\end{proof}

\begin{Lemma}\label{kV4AuslanderReiten}
Let $n$ be a positive integer. The following hold.

\smallskip\noindent (i)
The endomorphism $\underline\varphi^n_1$ in $\Endbar_{kV_4}(M_n)$ is an almost vanishing
morphism.

\smallskip\noindent (ii)
For any isomorphism $\alpha : M_n\cong\Omega(M_n)$ in $\modbar(kV_4)$
We have  $\alpha\circ\underline\varphi^n_1 = $
$\Omega(\underline\varphi^n_1)\circ\alpha$.
\end{Lemma}

\begin{proof}
Since $M_n$ has period $1$, any almost vanishing morphism starting at
$M_n$ is an endomorphism of $M_n$.
Let $\psi:M_n\to M_n$ be a representative of an almost vanishing endomorphism.
We use the notation of the previous lemma, except that we drop the superscripts of basis
elements.
First compose $\psi$ with $\zeta_n^{n+1}$. Given that $\underline\psi$ is almost
vanishing and $\zeta_n^{n+1}$ is not an split monomorphism, $\zeta_n^{n+1}\circ\psi$ has
to be a zero morphism in $\modbar(kV_4)$. But any $kV_4$-homomorphism between
indecomposable non projective modules
that factors through a projective module has its image contained in the socle of the
target module. As $\zeta_n^{n+1}$ is a monomorphism, it implies that the image of $\psi$
has to be contained in $\soc(M_n)$.
Using Lemma \ref{zeroendomorphisms} (iv), $\psi$ can be chosen in the $k$-vector space
generated by $\{\varphi^n_i\,|\,1\le i\le n\}$. So there exist $\alpha_i\in k,\,1\le i\le
n$ such that $\psi=\sum_{i=1}^n \alpha_i \varphi^n_i$. The aim is to prove that
$\alpha_i=0$ for all $2\le i\le n$.
To do that we pre-compose $\psi$ with $\zeta_j^n$ where $j=n-i+1$. Again the class of
this composition has to be a zero morphism in $\modbar(kV_4)$.
The proof is by reverse induction on $i$, which amounts to proceeding
by induction over $j=n-i+1$.
If $j=1$ then $\psi\circ\zeta_1^n(a_1)=\alpha_n\varphi_n^n(a_n)=\alpha_n b_1$,
where the notation of basis elements of $M_n$ is as at the beginning of the
proof of Lemma \ref{zeroendomorphisms}. Pre-composing with $\xi_1^n$ one gets an
endomorphism $\psi\circ\zeta_1^n\circ\xi_1^n$ of $M_n$ that is equal to
$\alpha_n\varphi^n_1$. Since the class of this morphism has to be $0$ in $\modbar(kV_4)$
we get that $\alpha_n=0$.
Suppose now that $\alpha_i=0$ for all $i>n-j+1$. Pre-compose $\psi$ with $\zeta_j^n$; as
before the class of $\psi\circ\zeta_j^n$  in $\modbar(kV_4)$ is zero. By direct
computation $\psi\circ\zeta_j^n(a_1)=$ $\alpha_{n-j+1}\varphi^n_{n-j+1}(a_{n-j+1})=$
$\alpha_{n-j+1} b_1$ and $\psi\circ\zeta_j^n(a_i)=$
$\alpha_{n-j+1}\varphi^n_{n-j+i}(a_{n-j+1})=0$. Pre-compose now with $\xi_j^n$ and get an
endomorphism $\psi\circ\zeta_j^n\circ\xi_j^n$ of $M_n$ that is equal to
$\alpha_{n-j+1}\varphi^n_1$. The class of the latter morphism is zero in $\modbar(kV_4)$
if and only if $\alpha_{n-j+1}=0$.
The induction procedure stops at $j=n-1$, or equivalently, at $i = 2$,
whence (i). In order to prove (ii), note first that of the statement holds
for some isomorphism $\alpha$ it holds for any such isomorphism because
$\End_{kV_4}(M_n)$ is split slocal and $\underline\varphi^n_1$ is
almost vanishing. We use the notation introduced at the beginning of
the proof of Lemma \ref{zeroendomorphisms}; in particular, we have an injective
homomorphism $\iota : M_n\to (kV_4)^n$ and a surjective homomorphism
$\pi : (kV_4)^n\to M_n$. Define $\theta:(kV_4)^n\to(kV_4)^n$  by $\theta(s_1)=xs_1$ and
$\theta(s_j)=0$ for $2\le j\ne n$. A straightforward verification shows that the
following diagram commutes:
$$\xymatrix{
0\ar[r] & \Ml_n\ar[r]^\iota\ar[d]_{\varphi^n_1} &(kV_4)^n\ar[r]^\pi\ar[d]_\theta
&\Ml_n\ar[r]\ar[d]_{\varphi^n_1} &0\\
0\ar[r] & \Ml_n\ar[r]^\iota &(kV_4)^n\ar[r]^\pi              &\Ml_n\ar[r]                 
 &0
}$$
This shows that for some - and hence any - identification $M_n=\Omega(M_n)$
we have $\Omega(\underline\varphi^n_1) = \underline\varphi^n_1$, which completes
the proof.
\end{proof}

In conjunction with the remarks at the beginning of this section, this shows
that for any even dimensional indecomposable non projective $kV_4$-module $M^\lambda_n$
there exists, up to multiplication by a scalar, a unique element of the graded center in
any degree that is an almost vanishing endomorphism of  $M^\lambda_n$ and zero on all
indecomposable modules not isomorphic to $M^\lambda_n$. Now let $\eta$ be a natural
transformation belonging to the kernel $\CI$
of the canonical evaluation map $Z^*(\modbar(kV_4))\to \hat H^*(V_4;k)$.
Then $\eta$ is zero when evaluated at the trivial module, hence zero at any odd
dimensional inedcomposable $kV_4$-module, as these are precisely the $\Omega$-orbit of
$k$, up to isomorphism. In order to conclude the
proof of Theorem \ref{kleinfour} we need to show that for any even-dimensional
indecomposable $kV_4$-module $M$ the evaluation $\eta(M)$ is either zero
or almost vanishing. Thanks to the remarks at the beginning of this section, it
suffices to show this for the modules $M_n$, for $n$ a positive integer,
and this is what the following lemma does:

\begin{Lemma}
Let $\eta\in\CI$ and $n$ a positive integer.
Then $\eta(M_n)$ is a scalar multiple of $\underline\varphi^n_1$.
\end{Lemma}

\begin{proof}
The proof is by induction on $n$.
For $n=1$ the assertion is trivial. Suppose that $n>1$ and that $\eta_{M_l}$ is a
multiple of $\underline\varphi^l_1$ for all $1\le l<n$.  By Lemma
\ref{zeroendomorphisms}, there are coefficients $\alpha_i\in k$ such that
$\eta(M_n)=\sum_{i=1}^n \alpha_i \underline\varphi^n_i$. We prove that, for any for any
$k$ such that $n\geq k>1$, if $\eta(M_n) = \sum_{i=1}^k \alpha_i \underline\varphi^n_i$
then $\alpha_k=0$.
The following diagram in $\modbar(kV_4)$ is commutative:
$$\xymatrix{
\Ml_n\ar[r]^{\underline\xi_{n-k+1}^n\,\,}\ar[d]_{\eta(M_n)} &
\Ml_{n-k+1}\ar[d]^{\,\eta(M_{n-k+1})}\\
\Ml_n\ar[r]_{\underline\xi_{n-k+1}^n\,\,} & \Ml_{n-k+1}
}$$
By induction $\eta(M_{n-k+1})$ is an almost vanishing morphism or zero, so on one hand,
$\eta(\Ml_{n-k+1})\circ\underline\xi_{n-k+1}^n$ is zero in $\modbar(kV_4)$. On the other
hand,
an easy verification shows that
$$\underline\xi_{n-k+1}^n\circ\eta(\Ml_n)\circ\underline\zeta_{n-k+1}^n \ =\
\alpha_k\underline\varphi^{n-k+1}_1$$ This forces $\alpha_k= 0$, whence the result.
\end{proof}

Combining the above lemmas yields a proof of Theorem \ref{kleinfour}.

\section{Rank 1}

For completeness we include some remarks on the graded center
of the stable category of a finite $p$-group $P$ of
rank $1$. If $k$ is a field of odd characteristic $p$
then $P$ is cyclic, the algebra $kP$ has finite
representation type, and the structure of $\modbar(kP)$ is
calculated in \cite{KrYe} and \cite{KeLi}. If $p=2$
and $Q$ is a non cyclic finite $2$-group of rank $1$ then
$Q$ is a generalised quaternion group.

\begin{Theorem} \label{Qstablecenter}
Let $k$ be an algebraically closed field of characteristic $2$ and
let $Q$ be a generalised quaternion $2$-group.
The kernel $\CI$ of the canonical map $Z^*(\modbar(kQ)) \longrightarrow \hat H^*(Q;k)$
given by evaluation at the trivial $kQ$-module is infinite dimensional in odd degrees.
\end{Theorem}

The idea of the proof of Theorem \ref{Qstablecenter} is to reduce this to the
case where $Q = Q_8$ is quaternion of order $8$, and then to show that the
$2$-dimensional modules over the Klein four quotient group $V_4$ of $Q_8$ yield infinitely
many isomorphism classes of $kQ_8$-modules of period $2$ with the property that
$\Omega^2$ acts trivially on the almost vanishing morphisms starting at these
modules; this implies that
the almost vanishing morphisms determine infinitely many linearly independent
elements in each odd degree of the graded center of the
stable module category. In what follows, $k$ is a (not necessarily algebraically closed)
field of characteristic $2$ and $Q_8$  a quaternion group of order $8$, with generators
$s$, $t$ of order $4$ satisfying $s^2 = t^2$ and $tst^{-1} = s^{-1}$. Set $z = s^2$; this
is the unique (and hence central) involution of $Q_8$. We start with a few elementary
observations.

\begin{Lemma} \label{kQ8center}
We have $kQ_8(z-1)\subseteq Z(kQ_8)$.
\end{Lemma}

\begin{proof}
We have $s(z-1) = s^3 + s$; this is the conjugacy class sum of $s$, hence
contained in $Z(kQ_8)$. The same argument applies to all elements of order
$4$ in $Q_8$, whence the result.
\end{proof}

\begin{Lemma} \label{cyclicshifted4}
Let $(\lambda,\mu)\in k^2$ and set $x = \lambda (s-1) + \mu (t-1)$.
Suppose that $\lambda^2+\lambda\mu+\mu^2\neq 0$.
Then $1+x$ is invertible of order $4$ in $(kQ_8)^\times$, and $kQ_8$ is
projective as left or right $k\langle 1+x \rangle$-module.
\end{Lemma}

\begin{proof}
We have $x^2 = (\lambda^2+\mu^2)(z-1) + \lambda\mu(s-1)(t-1) + \lambda\mu(t-1)(s-1)$. A
short calculation shows that
$(s-1)(t-1)+(t-1)(s-1) = st(z-1)$. This is an element in $Z(kQ_8)$, by
Lemma \ref{kQ8center}. Thus $x^2 = u(z-1)$, where $u = (\lambda^2+\mu^2)\cdot 1 +
\lambda\mu\cdot(st)$. The hypotheses on
$\lambda$ and $\mu$ imply that $u$ is invertible. Thus $\langle 1+x\rangle$
has order $4$ and Lemma \ref{cyclicmodify}
implies that $kQ_8$ is projective as $k\langle 1+x^2\rangle$ module. Since
$\langle 1+x^2\rangle$ is the unique subgroup of order $2$ of $\langle 1+x \rangle$
it follows that $kQ_8$ is projective as left or right $k\langle 1+x\rangle$-module.
\end{proof}

The hypothesis $\lambda^2+\lambda\mu+\mu^2\neq 0$ in the previous lemma
means that at leat one of $\lambda$, $\mu$ is nonzero, and if they are both
nonzero then $\frac{\lambda}{\mu}$ is not a cube root of unity.

\begin{Proposition} \label{Q8period2}
Let $(\lambda,\mu)\in k^2$ and set $x = \lambda (s-1) + \mu (t-1)$.
Suppose that $\lambda^2+\lambda\mu+\mu^2\neq 0$.
Set $M = kQ_8(z-1)x$ and $N = kQ_8x$. Then $M$, $N$ are absolutely indecomposable, and
the following hold.

\smallskip\noindent (i) We have isomorphisms
$\alpha : M \cong \Omega^2(M)$ and $\beta : \Omega(M) \cong N$
in $\modbar(kQ_8)$.

\smallskip\noindent (ii) We have
$M \subseteq N^{\langle z\rangle} = kQ_8(z-1)$.

\smallskip\noindent (iii)
For any $\varphi\in\Hom_{kQ_8}(M,N)$ there is
an element $c\in kQ_8$ such that $\varphi(m) = cm$ for all $m\in M$;
in particular, $\Im(\varphi)\subseteq M$.

\smallskip\noindent (iv) Set $\gamma = \Omega(\beta)\circ\Omega(\alpha)\circ\beta^{-1} :
N\cong \Omega^2(N)$. Then, for any $\varphi\in\Hom_{kQ_8}(M,N)$ we have
$\Omega^2(\underline\varphi) \circ \alpha = \gamma\circ\underline\varphi$.
\end{Proposition}

\begin{proof} The modules $M$, $N$ are cyclic modules over the split local algebra
$kQ_8$, hence absolutely indecomposable.
Since $x^2 = u(z-1)$ for some $u\in (kQ_8)^\times$ we have
$x^2(z-1) = 0$. Thus $N = kQ_8x$ is contained in the kernel of the surjective
homomorphism $kQ_8\rightarrow M$ given by right multiplication with $(z-1)x$, hence equal
to this kernel because both have dimension $6$. The same argument
shows that the endomorphism $kQ_8 \rightarrow kQ_8$ given by right multiplication with
$x$ has kernel equal to $M$. This proves that for this choice of projective covers of $M$
and $N$ we get $\Omega(M) = N$ and
$\Omega^2(M) = M$, whence (i).
The inclusion $M\subseteq N$ is trivial. Since $z-1$ annihilates $M$, we have $M
\subseteq N^{\langle z\rangle} \subseteq kQ_8(z-1)$, where the
second inclusion uses the fact that $kQ_8$ is projective as right $k\langle
z\rangle$-module. Clearly also
$\Im(\varphi) \subseteq N^{\langle x\rangle}$. Now  $kQ_8$ is projective as a right
$k\langle 1+x\rangle$-module by \ref{cyclicshifted4}, thus $kQ_8$ is free of rank $2$ as
a $k\langle 1+x \rangle$-module. It follows that $kQ_8x$, when
viewed as right $k\langle 1+x\rangle$-module, is isomorphic to the direct sum
of two copies of the unique (up to isomorphism) $3$-dimensional
$k\langle 1+x\rangle$-module. The restriction of this $3$-dimensional module
to $\langle 1+x^2\rangle$ is of the form $k\oplus k\langle 1+x^2\rangle$, so
the space annihilated by $x^2$ has dimension $2 + 2 = 4$. Since $x^2 = u(z-1)$
for some invertible element $u$ it follows that the subspace of $N$
annihilated by $z-1$ has also dimension $4$, hence is equal to $kQ_8(z-1)$.
This proves (ii).
Let $\varphi\in\Hom_{kQ_8}(M,N)$.
It follows from (ii) that $\Im(\varphi) \subseteq kQ_8(z-1)$.
Thus we may view $\varphi$ as a homomorphism from the submodule $M$ of
$kQ_8(z-1)$ to $kQ_8(z-1)$. Both
modules are annihilated by $z-1$, so we may view this as a homomorphism of
$kV_4$-modules, where $V_4 = Q_8/\langle z\rangle$. But since $M$ is
a submodule of $kQ_8(z-1)\cong kV_4$ as $kV_4$-modules and since
$kV_4$ is commutative it follows that $\varphi$ is actually induced
by {\it left} multiplication with an element in $kV_4$. Pulling this back
to $kQ_8$ shows that $\varphi$, as a homomorphism of $kQ_8$-modules, is induced
by left multiplication on $M$ with some element $c\in kQ_8$, composed
with the inclusion $M\subseteq N$. This proves (iii). In order to prove (iv), we choose
notation such that $\alpha$, $\beta$ are
equalities (that is, we identify $N = \Omega(M)$ and $M = \Omega^2(M)$
in the obvious way described at the beginning of this proof), and consider a commutative
diagram of $kQ_8$-modules
of the form
$$\xymatrix{
0\ar[r] & M\ar[r]^\iota\ar[d]_{\varphi'} &kQ_8\ar[r]^{x}\ar[d]_{v}
&kQ_8\ar[r]^{x(z-1)}\ar[d]_{w} &M\ar[r]\ar[d]_{\varphi} &0\\
0\ar[r] & N\ar[r]_\kappa                 &kQ_8\ar[r]_{x(z-1)} &kQ_8\ar[r]_{x} &N\ar[r]   
             &0
}$$
where $\iota$, $\kappa$ are the inclusions, $v$, $w$ are suitable elements in $kQ_8$ and
where
an arrow labelled by an element in $kQ_8$ means the homomorphism induced by right
multiplication with that element. We need to show $\underline\varphi' =
\underline\varphi$. For $a\in kQ_8$, the commutativity
of the right square in this diagram means that $\varphi(ax(z-1)) = awx$.
The commutativity
of the middle square is equivalent to $axw = avx(z-1)$, hence $xw = vx(z-1)$. The
commutativity of the left square implies that $\varphi'(ax(z-1)) = $ $ax(z-1)v =$
$avx(z-1) =axw$, where we used that $x(z-1)\in Z(kQ_8)$ by Lemma \ref{kQ8center}. In
order to show that $\varphi=\varphi'$
it suffices to show that $w$ can be chosen in $Z(kQ_8)$. By (iii) there is an element
$c\in kQ_8$ such that
$\varphi'(x(z-1)) = cx(z-1) = xc(z-1)$, and hence we may take
$w = c(z-1)$, which is an element in $Z(kQ_8)$ by Lemma \ref{kQ8center}.
\end{proof}

\begin{proof} [Proof of Theorem \ref{Qstablecenter}]
We use the notation and assumptions of Theorem \ref{Qstablecenter};
in particular, $k$ is now algebraically closed.
We consider first the case where $Q = Q_8$.
It follows from Proposition \ref{isomod} applied to the Klein four group
$V_4 = Q_8/\langle z\rangle$ and from Proposition \ref{Q8period2} (i)
that if $(\lambda,\mu)\in k^2$ runs over a set of
representatives of elements in $\P^1(k)$ satisfying $\lambda^2+\lambda\mu+\mu^2\neq 0$
then $M_{\lambda,\mu} = kQ_8(z-1)(\lambda (s-1)+\mu (t-1))$ runs over pairwise
nonisomorphic indecomposable $kQ_8$-modules of period $2$.
It follows from Proposition \ref{Q8period2} (iv) that the almost vanishing morphisms
starting at these modules determine elements in the graded center of the stable module
category of $kQ_8$ in any odd degree, whence the result in the case $Q = Q_8$. For $Q$ a
generalised quaternion $2$-group we may identify $Q_8$ to a subgroup of $Q$, and then the
induction functor $\Ind^Q_{Q_8}$ sends the indecomposable modules of period $2$
considered in Proposition \ref{Q8period2} to indecomposable $kQ$-modules of period $2$.
Moreover, it follows from Proposition \ref{Q8period2} (iv) in conjunction with Lemma
\ref{Omegainvariance}
that the almost vanishing morphisms starting at these modules are
again invariant under $\Omega_{kQ}^2$, hence define elements in the
graded center of the stable category of $kQ$-modules.
\end{proof}

\end{document}